\documentclass[12pt,a4]{amsart}
\usepackage{amsmath}
\usepackage{amssymb}
\usepackage{amsthm}
\usepackage{enumitem}
\usepackage{url}
\usepackage{soul}
\usepackage{times}
\usepackage[T1]{fontenc}
\usepackage{color}
\usepackage{dsfont}
\usepackage{epsfig}
\usepackage[hidelinks]{hyperref}
\usepackage{setspace}
\usepackage{epstopdf}
\usepackage[authoryear,round,sort]{natbib}
\usepackage{comment}
\usepackage{booktabs}
\usepackage{float}
\usepackage{algorithm}
\usepackage{stmaryrd}
\usepackage{subcaption,graphicx}
\usepackage{float}
\numberwithin{equation}{section}

\usepackage[font=small,labelfont=bf]{caption}
\DeclareCaptionFont{tiny}{\tiny}
\usepackage{geometry}
\theoremstyle{plain}
\newtheorem{proposition}{Proposition}[section]

\newtheorem{theorem}{Theorem}[section]
\newtheorem{lemma}{Lemma}[section]

\theoremstyle{definition}
\newtheorem{definition}{Definition}[section]

\theoremstyle{remark}
\newtheorem{rk}{Remark}[section]
\expandafter\let\expandafter\oldproof\csname\string\proof\endcsname
\let\oldendproof\endproof
\renewenvironment{proof}[1][\proofname]{%
  \oldproof[\noindent\textbf{#1.} ]%
}{\oldendproof}

\newcommand{\E}{\mathbb{E}}

\newcommand{\be}{\begin{equation}}
\newcommand{\ee}{\end{equation}}
\newcommand{\by}{\begin{eqnarray*}}
\newcommand{\ey}{\end{eqnarray*}}

\renewcommand{\leq}{\leqslant}
\renewcommand{\geq}{\geqslant}
\usepackage{xcolor}
\definecolor{dark-red}{rgb}{0.4,0.15,0.15}
\definecolor{dark-blue}{rgb}{0.15,0.15,0.4}
\definecolor{medium-blue}{rgb}{0,0,0.5}
\hypersetup{
    colorlinks, linkcolor={red},
    citecolor={blue}, urlcolor={blue}
}
\allowdisplaybreaks

\begin{document}
%\setstretch{1.3}
\title{On hitting time, mixing time and geometric interpretations of Metropolis-Hastings reversiblizations}
\author{Michael C.H. Choi}
\address{Institute for Data and Decision Analytics, The Chinese University of Hong Kong, Shenzhen, Guangdong, 518172, P.R. China.}
\email{michaelchoi@cuhk.edu.cn}
\author{Lu-Jing Huang}
\address{College of Mathematics and Informatics, Fujian Normal University, Fuzhou, Fujian, 350108, P.R. China}
\email{huanglj@fjnu.edu.cn}
%\address{School of Operations Research and Information Engineering, Cornell University, Ithaca, New York}
%\email{cc2373@cornell.edu}
\date{\today}
\maketitle

%\begin{abstract}
%	We study two types of Metropolis-Hastings reversiblizations for non-reversible Markov chains. Consider a Markov kernel $P$ with stationary measure $\pi$ and its time-reversal denoted by $P^*$. Inspired by the classical Metropolis transition kernel $M_1$, we introduce a self-adjoint kernel $M_2$ that captures the opposite transition effect of $M_1$. This permits us to write $P+P^* = M_1 + M_2$, which leads to bounds on the spectral gap of $P$ in terms of $M_1$ and $M_2$. We obtain an expansion of $P$ (and $P^*$) in terms of the spectral measures of $M_1$ and $M_2$. In the spirit of \cite{Fill91} and \cite{Paulin15}, we introduce a new pseudo-spectral gap based on $M_1$ and $M_2$, and show that the total variation distance from stationarity can be bounded by this gap. We give variance bounds of the Markov chain in terms of the gap.
%\end{abstract}

\begin{abstract}
	Given a target distribution $\mu$ and a proposal chain with generator $Q$ on a finite state space, in this paper we study two types of Metropolis-Hastings (MH) generator $M_1(Q,\mu)$ and $M_2(Q,\mu)$ in a continuous-time setting. While $M_1$ is the classical MH generator, we define a new generator $M_2$ that captures the opposite movement of $M_1$ and provide a comprehensive suite of comparison results ranging from hitting time and mixing time to asymptotic variance, large deviations and capacity, which demonstrate that $M_2$ enjoys superior mixing properties than $M_1$. To see that $M_1$ and $M_2$ are natural transformations, we offer an interesting geometric interpretation of $M_1$, $M_2$ and their convex combinations as $\ell^1$ minimizers between $Q$ and the set of $\mu$-reversible generators, extending the results by \cite{BD01}. We provide two examples as illustrations. In the first one we give explicit spectral analysis of $M_1$ and $M_2$ for Metropolised independent sampling, while in the second example we prove a Laplace transform order of the fastest strong stationary time between birth-death $M_1$ and $M_2$.
	\smallskip
	
	
\end{abstract}

%\tableofcontents

%\newpage

\section{Introduction}

In this paper, we study the so-called Metropolis-Hastings reversiblizations in a continuous-time and finite state space setting. This work is largely motivated by \cite{Choi16}, in which the author introduced two Metropolis-Hastings (MH) kernels $M_1$ and $M_2$ to study non-reversible Markov chains in discrete-time. While $M_2$ is a self-adjoint kernel, $M_2$ may not be Markovian, which makes further probabilistic analysis of $M_2$ to be difficult. In a continuous-time setting however, we will show that similar construction for $M_2$ still gives a valid Markov generator, and this observation motivates us to study fine theoretical properties of $M_2$. This paper is therefore devoted to the study of $M_1$ and $M_2$ and offers relevant comparison results between $M_1$, $M_2$ and the proposal chain. It turns out that $M_2$ enjoys superior hitting time and mixing time properties when compared with $M_1$, and so from a Markov chain Monte Carlo perspective, $M_2$ offers acceleration when compared with the classical MH algorithm $M_1$.

%The study of non-reversible Markov chains is fascinating from both theoretical and applied perspective. On the theoretical side, the major technical challenge lies in analyzing non-self-adjoint operators as many classical results for reversible Markov chains break down in this setting. From an application point of view, non-reversible Markov chains arise naturally in a variety of domains such as queueing networks \cite{FMM08}, and non-reversible Markov chain Monte Carlo (MCMC) \cite{CH13,CCHP12,RR15,SGS10,Bie16,HM17}. In the literature, there are a few techniques that have been developed for analyzing such non-reversible Markov chains. This includes dilation \cite{Kendall59}, reversiblizations \cite{Fill91,Paulin15,Choi16} or intertwining \cite{Patie-Savov}, to name but a few.

The rest of this paper is organized as follow. In Section \ref{sec:M1M2}, we fix our notations and define the two MH generators $M_1$ and $M_2$ that we study throughout our paper. The main results can be found in both Section \ref{sec:geom} and Section \ref{sec:hitmixcompare}. In Section \ref{sec:geom} we provide interesting geometric interpretations of $M_1$, $M_2$ and their convex combinations as $\ell^1$ minimizers between the proposal chain and the set of generator that are reversible with respect to the target distribution. In Section \ref{sec:hitmixcompare} we compare various hitting time and mixing time parameters between $M_1$ and $M_2$. The final section is devoted to two concrete examples. More specifically, in the first example we consider the special case of Metropolised independent sampling in Section \ref{sec:MIS} and offer explicit spectral analysis for $M_1$ and $M_2$, while in the second example in Section \ref{sec:bd} we study birth-death proposal chain that allows for effective comparison of the fastest strong stationary time of $M_1$ and $M_2$.

\section{Metropolis-Hastings kernels: $M_1$ and $M_2$}\label{sec:M1M2}

In this section, we give the construction of continuous-time Metropolis-Hastings (MH) Markov chains. To fix our notation, we let $\mathcal{X}$ be a finite state space and $\mu$ be a target distribution on $\mathcal{X}$. It is perhaps well-known that the classical MH algorithm offers a way to construct a discrete-time Markov chain that is reversible with respect to $\mu$. For pointers on this subject, we refer readers to \cite{MRRTT53,RR04} and the references therein. Here we adapt the basic idea and recast the classical discrete-time MH algorithm to a continuous-time setting so as to construct what we call the first MH Markov chain. We note that similar construction of continuous-time Metropolis-type algorithms can be found in \cite{DM09}.

%In this paper, we study ergodic continuous-time Markov chain $X = (X_t)_{t \geq 0}$ on a countable state space $\mathcal{X}$ with transition semigroup $(P^t)_{t \geq 0}$, infinitesimal generator $Q = (Q(x,y))_{x,y \in \mathcal{X}}$ and stationary distribution $\pi = (\pi(x))_{x \in \mathcal{X}}$. Let $\ell^2(\pi)$ be the Hilbert space of real square summable functions with respect to $\pi$, endowed with the usual inner product given by, for $f,g \in \ell^2(\pi)$,
%$$\langle f,g \rangle_{\pi} = \sum_{x \in \mathcal{X}} f(x) g(x) \pi(x).$$
%We denote by $\ell^2_0(\pi)$ to be the space of mean zero function with respect to $\pi$, that is, $\ell^2_0(\pi) = \{f \in \ell^2(\pi);~\pi(f) = 0\}$. We write $Q^* = (Q^*(x,y))_{x,y \in \mathcal{X}}$ to be the adjoint operator of $Q$ in $\ell^2(\pi)$. $Q$ is said to be $\pi$-reversible if and only if $Q$ is a self-adjoint operator in $\ell^2(\pi)$, that is, $Q = Q^*$.

\begin{definition}[The first MH generator]\label{def:M1}
Given a target distribution $\mu$ on finite state space $\mathcal{X}$ and a proposal continuous-time irreducible Markov chain with generator $Q$, the first MH Markov chain has generator given by $M_1 = M_1(Q,\mu) = (M_1(x,y))_{x,y \in \mathcal{X}}$, where
$$M_1(x,y) := \begin{cases} \min\left\{Q(x,y),\dfrac{\mu(y)}{\mu(x)}Q(y,x)\right\}, &\mbox{if } x \neq y; \\
- \sum_{z: z \neq x} M_1(x,z), & \mbox{if } x = y. \end{cases}$$
\end{definition}

Note that the above definition closely resembles the classical MH algorithm, in which we simply substitute transition probability in the MH algorithm by the transition rate $Q$ of the proposal chain. By mirroring the transition effect of $M_1$ and capturing the opposite movement, we can construct another MH generator, which is what we  call the second MH generator. More precisely, we give a definition for it as follows.

\begin{definition}[The second MH generator]\label{def:M2}
	Given a target distribution $\mu$ on finite state space $\mathcal{X}$ and a proposal continuous-time irreducible Markov chain with generator $Q$, the second MH Markov chain has generator given by $M_2 = M_2(Q,\mu) = (M_2(x,y))_{x,y \in \mathcal{X}}$, where
	$$M_2(x,y) := \begin{cases} \max\left\{Q(x,y),\dfrac{\mu(y)}{\mu(x)}Q(y,x)\right\}, &\mbox{if } x \neq y; \\
	- \sum_{z:z \neq x} M_2(x,z), & \mbox{if } x = y. \end{cases}$$
\end{definition}

Comparing Definition \ref{def:M1} and \ref{def:M2}, we see that in the former we take $\min$ while in the latter we consider $\max$ for off-diagonal entries. It is what we meant by $M_2$ mirroring the transition effect of $M_1$. As another remark, we note that in the discrete-time setting, $M_2$ as defined in \cite{Choi16} may not be a Markov kernel. In the continuous-time setting however, $M_2$ as defined in Definition \ref{def:M2} is a valid Markov generator.

To allow for effective comparison between these generators, we now introduce the notion of Peskun ordering of continuous-time Markov chains. This partial ordering was first introduced by \cite{Pesk73} for discrete time Markov chains on finite state space. It was further generalized by \cite{Tie98} to general state space, and by \cite{LM08} to continuous-time Markov chains.

\begin{definition}[Peskun ordering]
	Suppose that we have two continuous-time Markov chains with generators $Q_1$ and $Q_2$ respectively. Both chains share the same stationary distribution $\pi$. $Q_1$ is said to dominate $Q_2$ off-diagonally, written as $Q_1 \succeq Q_2$, if for all $x \neq y \in \mathcal{X}$, we have
	$$Q_1(x,y) \geq Q_2(x,y).$$
\end{definition}

For a given target distribution $\mu$ and proposal generator $Q$, define the time-reversal generator of $Q$ with respect to $\mu$ as
$$
Q^*(x,y)=\frac{\mu(y)Q(y,x)}{\mu(x)},\quad x,y\in \mathcal{X}.
$$
$Q$ is said to be $\mu$-reversible if and only if $Q=Q^*$. For convenience, let $\bar{Q}=(Q+Q^*)/2$. We also denote the inner product with $\mu$ by $\langle\cdot, \cdot \rangle_\mu$, that is, for any functions $f,\ g:\ \mathcal{X}\rightarrow\mathbb{R}$,
$$
\langle f,g\rangle_\mu=\sum_{x\in \mathcal{X}}f(x)g(x)\mu(x).
$$
In the following, we collect a few elementary observations and results on the behaviour of generators $Q,\ Q^*,\ M_1$ and $M_2$.

\begin{lemma}\label{lem:M1M2}
	Given a target distribution $\mu$ on $\mathcal{X}$ and proposal chain with generator $Q$, then we have
	\begin{enumerate}
		\item $M_1$ and $M_2$ are $\mu$-reversible. \label{it:1}
		\item(Peskun ordering) $M_2 \succeq M_1$. \label{it:2}
		\item \label{it:3} For any function $f:\ \mathcal{X}\rightarrow \mathbb{R}$,
		$$\langle M_2 f,f \rangle_{\mu} \leq \langle M_1 f,f \rangle_{\mu}.$$
	\end{enumerate}
	If we take $\mu = \pi$, the stationary distribution of the proposal chain, then we have
	\begin{enumerate}[resume]
		\item $Q + Q^* = M_1 + M_2$. \label{it:4}
		\item(Peskun ordering) $M_2 \succeq Q \succeq M_1$. \label{it:5}
		\item For any function $f:\ \mathcal{X}\rightarrow \mathbb{R}$,
			  $$\langle M_2 f,f \rangle_{\pi} \leq \langle Q f,f \rangle_{\pi} \leq \langle M_1 f,f \rangle_{\pi}.$$ \label{it:6}
	\end{enumerate}
\end{lemma}

\begin{proof}

For item \eqref{it:1}, it is easy to see that for $x \neq y$,
$$
\mu(x)M_2(x,y) = \max\{\mu(x)Q(x,y),\mu(y)Q(y,x)\} = \max\{\mu(y)Q(y,x),\mu(x)Q(x,y)\} = \mu(y)M_2(y,x).
$$
So $M_2$ is $\mu$-reversible. Similarly, the $\mu$-reversibility for $M_1$ can be derived via replacing $\max$ by $\min$ in the above argument.

	Next, we prove item \eqref{it:2}, which trivially holds since
	$$
M_2(x,y) =  \max\left\{Q(x,y),\dfrac{\mu(y)}{\mu(x)}Q(y,x)\right\} \geq  \min\left\{Q(x,y),\dfrac{\mu(y)}{\mu(x)}Q(y,x)\right\} = M_1(x,y).
$$
Item \eqref{it:3} follows readily from \cite[Theorem $5$]{LM08} since both $M_1$ and $M_2$ are $\mu$-reversible. Next, we prove item \eqref{it:4}. We see that
	$$Q(x,y) + Q^*(x,y) = \min\left\{Q(x,y),Q^*(x,y)\right\} + \max\left\{Q(x,y),Q^*(x,y)\right\} = M_1(x,y) + M_2(x,y),$$
if $\mu=\pi$. We proceed to prove item \eqref{it:5}, which follows from
	$$M_2(x,y) = \max\left\{Q(x,y),Q^*(x,y)\right\} \geq Q(x,y) \geq \min\left\{Q(x,y),Q^*(x,y)\right\} = M_1(x,y).$$
	Finally, we prove item \eqref{it:6}. For any function $f$, we see that
	$$\langle Q f,f \rangle_{\pi} = \left\langle \bar{Q} f,f \right\rangle_{\pi}.$$
	As we have $M_2 \succeq \bar{Q} \succeq M_1$ and they are all reversible generators, desired results follow from \cite[Theorem $5$]{LM08}.
\end{proof}

The above lemma will be frequently exploited to develop comparison results in Section \ref{sec:hitmixcompare}.

\section{Geometric interpretation of $M_1$ and $M_2$}\label{sec:geom}

This section is devoted to offer a geometric interpretation for both $M_1$ and $M_2$. Suppose that we are given a target distribution $\mu$ on $\mathcal{X}$ and a proposal chain with generator $Q$ and stationary distribution $\pi$. Our result is largely motivated by the work of \cite{BD01}, who is the first to study geometric consequences of $M_1$ in discrete-time. As we will show in our main result Theorem \ref{thm:geomM1M2} below, it turns out that both $M_1$ and $M_2$ (as well as their convex combinations) minimize certain $\ell^1$ distance to the set of $\mu$-reversible Markov generator on $\mathcal{X}$. As a result, they are natural transformations that maps a given Markov generator to the set of $\mu$-reversible Markov generators.

Let us now fix a few notations and define a metric to quantify the distance between two Markov generators. We write $\mathcal{R}(\mu)$ to be the set of conservative $\mu$-reversible Markov generators and $\mathcal{S}(\mathcal{X})$ to be the set of Markov generator on $\mathcal{X}$. For any $Q_1, Q_2 \in \mathcal{S}(\mathcal{X})$, we define a metric $d_{\mu}$ on $\mathcal{S}(\mathcal{X})$ to be
$$d_{\mu}(Q_1,Q_2) := \sum_{x \in \mathcal{X}} \sum_{y: x \neq y} \mu(x) |Q_1(x,y)-Q_2(x,y)|.$$
To see that $d_{\mu}$ defines a metric, we have $d_{\mu}(Q_1,Q_2) = 0$ implies $Q_1(x,y) = Q_2(x,y)$ for all off-diagonal entries and since each row sums to zero we have $Q_1(x,x) = Q_2(x,x)$ for all $x \in \mathcal{X}$. The distance between $Q$ and $\mathcal{R}(\mu)$ is then defined to be
\begin{align}\label{eq:l1metric}
d_{\mu}(Q,\mathcal{R}(\mu)) := \inf_{M \in \mathcal{R}(\mu)} d_{\mu}(Q,M).
\end{align}
With the above notations in mind, we are now ready to state our main result in this section:

\begin{theorem}\label{thm:geomM1M2}
	The convex combinations $\alpha M_1 + (1-\alpha)M_2$ for $\alpha \in [0,1]$ minimize the distance $d_{\mu}$ between $Q$ and $\mathcal{R}(\mu)$. That is,
	$$d_{\mu}(Q,\mathcal{R}(\mu))= d_{\mu}(Q,\alpha M_1 + (1-\alpha)M_2).$$
	Moreover, $M_1$ (resp.~ $M_2$) is the unique closest element of $\mathcal{R}(\mu)$ that is coordinate-wise no larger (resp.~no smaller) than $Q$ off-diagonally.
\end{theorem}

\begin{rk}
	Taking $\mu = \pi$ in Theorem \ref{thm:geomM1M2}, the stationary distribution of $Q$, $\alpha = 1/2$ and using $Q + Q^* = M_1 + M_2$ by Lemma \ref{lem:M1M2}, we see that
	$$d_{\pi}(Q,\mathcal{R}(\pi)) = d_{\pi}(Q,\bar{Q}).$$
	Thus, the additive reversiblization $\bar{Q}$ is a natural transformation of $Q$ that minimizes the distance $d_{\pi}$ between $Q$ and the set of $\pi$-reversible generators.
\end{rk}

To illustrate this result, we consider the simplest possible case of two-state with $\mathcal{X} = \{0,1\}$ and
$$Q_{(a,b)} = \begin{bmatrix}
-a       & a \\
b       & -b
\end{bmatrix},$$
where $a,b > 0$. Thus this generator $Q_{(a,b)}$ can be parameterized as $(a,b)$ on $\mathcal{S}(\mathcal{X}) = \{Q_{(a,b)};~(a,b) \in \mathbb{R}^+ \times \mathbb{R}^+\}$. The intersection of $\mathcal{S}(\mathcal{X})$ and the line $\mu(1)b - \mu(0)a = 0$ is therefore the set of $\mu$-reversible generator $\mathcal{R}(\mu)$, that is
$\mathcal{R}(\mu) = \{Q_{(a,b)};~\mu(1)b - \mu(0)a = 0, a,b > 0\}$. These are illustrated in Figure \ref{fig:geometric} below.

In Figure \ref{fig:geometric}, there are two points $(a_1,b_1)$ and $(a_2,b_2)$. The former point lies above $\mathcal{R}(\mu)$ while the latter lies below the straight line $\mathcal{R}(\mu)$, and hence they represent two non-reversible Markov chains. We see that $M_1$ projects vertically for $(a_1,b_1)$ and horizontally for $(a_2,b_2)$. On the other hand, $M_2$ mirrors the action of $M_1$ and does the opposite: it projects horizontally for $(a_1,b_1)$ and vertically for $(a_2,b_2)$. In addition, we can compute the distance $d_{\mu}$ explicitly between these generators as in Theorem \ref{thm:geomM1M2}:
\begin{align*}
d_{\mu}(Q_{(a_1,b_1)}, M_1(Q_{(a_1,b_1)},\mu) ) &= |\mu(1) b_1 - \mu(0) a_1|, \\
d_{\mu}(Q_{(a_1,b_1)}, M_2(Q_{(a_1,b_1)},\mu) ) &= |\mu(0) a_1 - \mu(1) b_1|, \\
\end{align*}
and
\begin{align*}
&\quad d_{\mu}(Q_{(a_1,b_1)}, \alpha M_1(Q_{(a_1,b_1)},\mu) + (1-\alpha)M_2(Q_{(a_1,b_1)},\mu) )\\
&= (1-\alpha)|\mu(0) a_1 - \mu(1) b_1| + \alpha |\mu(1) b_1 - \mu(0) a_1| \\
&= |\mu(0) a_1 - \mu(1) b_1|,
\end{align*}
where $\alpha \in [0,1]$. From Theorem \ref{thm:geomM1M2}, they all minimize the distance between $Q_{(a_1,b_1)}$ and $\mathcal{R}(\mu)$.

\begin{figure}
	\includegraphics[width=0.7\linewidth]{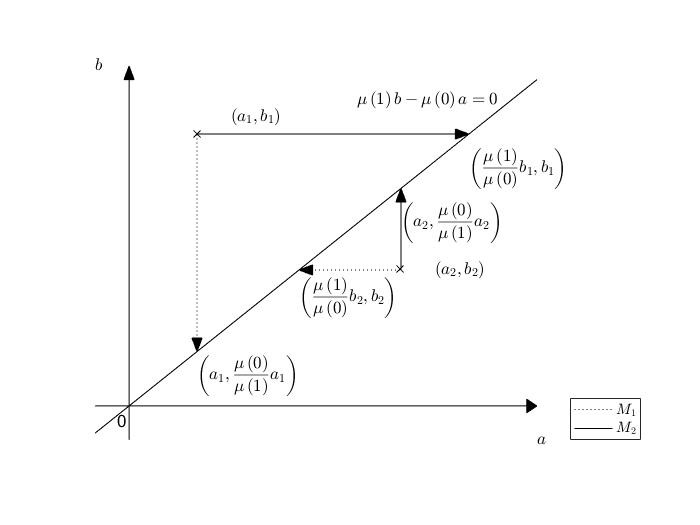}
	\caption{$M_1$ and $M_2$ as $\ell^1$ projections in the $2 \times 2$ case}
	\label{fig:geometric}
\end{figure}

We now proceed to give a proof of Theorem \ref{thm:geomM1M2}.

\begin{proof}[Proof of Theorem \ref{thm:geomM1M2}]
	The proof is inspired by the proof of Theorem $1$ in \cite{BD01}. We first define two helpful half spaces:
	\begin{align*}
		H^{<} = H^{<}(Q,\mu) &:= \big\{(x,y);~\mu(x)Q(x,y) < \mu(y)Q(y,x)\big\}, \\
		H^{>} = H^{>}(Q,\mu) &:= \big\{(x,y);~\mu(x)Q(x,y) > \mu(y)Q(y,x)\big\}.
	\end{align*}
	We now show that for $N \in \mathcal{R}(\mu)$, $d_{\mu}(Q,N) \geq d_{\mu}(Q,M_2)$. First, we note that
	\begin{align*}
	d_{\mu}(Q,N) \geq \sum_{(x,y) \in H^{<}} \big[\mu(x) |Q(x,y) - N(x,y)| + \mu(y) |Q(y,x) - N(y,x)|\big].
	\end{align*}
	As $N$ is $\mu$-reversible, setting $N(x,y) = Q(x,y) + \epsilon_{xy}$ gives $N(y,x) = \frac{\mu(x)}{\mu(y)}(Q(x,y) + \epsilon_{xy})$. Plugging these expressions back yields
	\begin{align*}
		d_{\mu}(Q,N) &\geq \sum_{(x,y) \in H^{<}} \Big[\mu(x) |\epsilon_{xy}| + \mu(y) \left|Q(y,x) - \frac{\mu(x)}{\mu(y)}(Q(x,y) + \epsilon_{xy})\right|\Big] \\
		&=  \sum_{(x,y) \in H^{<}} \big[\mu(x) |\epsilon_{xy}| + \left|\mu(y) Q(y,x) - \mu(x)Q(x,y) - \mu(x) \epsilon_{xy}\right|\big] \\
		&\geq \sum_{(x,y) \in H^{<}} \left|\mu(y) Q(y,x) - \mu(x)Q(x,y) \right| = d_{\mu}(Q,M_2),
	\end{align*}
	where we use the reverse triangle inequality $|a-b| \geq |a| - |b|$ in the second inequality. For uniqueness, if $N$ is off-diagonally no smaller than $Q$, then $\epsilon_{xy} \geq 0$. If anyone is strictly positive, we have $d_{\mu}(Q,N) > d_{\mu}(Q,M_2)$. Alternatively, the uniqueness can be seen by observing that if $N(x,y) > Q(x,y)$ then $N(y,x) > \dfrac{\mu(x)}{\mu(y)}Q(x,y)$. Similarly, we can show $d_{\mu}(Q,N) \geq d_{\mu}(Q,M_1)$ via substituting $H^{<}$ by $H^{>}$. To see that $d_{\mu}(Q,M_1) = d_{\mu}(Q,M_2)$, we have
	\begin{align*}
	d_{\mu}(Q,M_2) &= \sum_{(x,y) \in H^{<}} \left|\mu(y) Q(y,x) - \mu(x)Q(x,y) \right| \\
	&= \sum_{(y,x) \in H^{>}} \left|\mu(y) Q(y,x) - \mu(x)Q(x,y) \right| = d_{\mu}(Q,M_1).
	\end{align*}
	As for convex combinations of $M_1$ and $M_2$, we see that
	\begin{align*}
	d_{\mu}(Q,\alpha M_1 + (1-\alpha)M_2) &= (1-\alpha)\sum_{(x,y) \in H^{<}} \left|\mu(y) Q(y,x) - \mu(x)Q(x,y) \right| \\
	&\quad + \alpha \sum_{(x,y) \in H^{>}} \left|\mu(y) Q(y,x) - \mu(x)Q(x,y) \right|\\
	&= (1-\alpha) d_{\mu}(Q,M_2) + \alpha d_{\mu}(Q,M_1) = d_{\mu}(Q,M_1).
	\end{align*}
\end{proof}

\section{Performance comparisons between $\bar{Q}$, $M_1$ and $M_2$}\label{sec:hitmixcompare}

The evaluation of the performance of Markov chains depends on the comparison criterion. Popular comparison criteria that have appeared in the literature include mixing time and hitting times, spectral gap, asymptotic variance and large deviations, see e.g. \cite{Pesk73,RR97,CLP99,DHN00,GM00,SGS10,CH13,Bie16,HM17,HM18,FHY92,CCHP12,Hwang05} and references therein. In this section we give some comparison theorems based on these parameters for chains $\bar{Q}$, $M_1$ and $M_2$. Recall our setting that we are given a proposal irreducible chain $X = (X_t)_{t \geq 0}$ with generator $Q$, transition semigroup $(P_t)_{t \geq 0}$ and stationary distribution $\pi$. We are primarily interested in the behaviour of the following parameters:

\begin{itemize}
	\item(Hitting times) We write
\begin{align*}
\tau_{A} &= \tau_{A}(Q) := \inf\{t \geq 0; X_t \in A\},\\
\tau_{A}^+& = \tau_{A}^+(Q) := \inf\{t > 0; X_t \in A \text{ and there exists } s \in (0,t) \text{ such that } X_s \neq X_0\}
\end{align*}
to be the first hitting time and the first return time to the set $A \subseteq \mathcal{X}$ of chain $X$ respectively, and the usual convention of $\inf \emptyset = \infty$ applies. We also adapt the common notation that $\tau_y := \tau_{\{y\}}$ (resp.~$\tau_y^+ := \tau_{\{y\}}^+$) for $y \in \mathcal{X}$. The commute time $t_{com}^{x,y}$ between two states $x,\ y$ is
	$$t_{com}^{x,y} = t_{com}^{x,y}(Q) := \E_x(\tau_y(Q)) + \E_y(\tau_x(Q)).$$
	Another hitting time parameter of interest is the average hitting time $t_{av}$, which is defined to be
	$$t_{av} = t_{av}(Q,\pi) := \sum_{x,y} \mathbb{E}_x(\tau_y) \pi(x)\pi(y).$$
	In fact, $t_{av}$ equals to the sum of the reciprocals of the non-zero eigenvalues of $-Q$, and it also has close connection with the notion of strong ergodicity, see for example \cite{Mao04,CuiMao10}.
	
	\item(Total variation mixing time) For $\epsilon > 0$, we write the total variation mixing time $t_{mix}(\epsilon)$ to be
	$$t_{mix}(\epsilon) = t_{mix}(Q,\pi,\epsilon) := \inf\big\{t \geq 0; \sup_x ||P_t(x,\cdot) - \pi||_{TV} < \epsilon\big\},$$
	where for any probability measure $\nu$ and $\pi$ on $\mathcal{X}$, $||\nu - \pi||_{TV} := \frac{1}{2} \sum_x |\nu(x) - \pi(x)|$ is the total variation distance between these two measures. A commonly used metric is $t_{mix}(1/4)$ where we take $\epsilon = 1/4$.
	
	\item(Spectral gap) Denote the spectral gap of $Q$ to be
	$$\lambda_2 = \lambda_2(Q,\pi) := \inf\big\{\langle -Qf,f \rangle_{\pi}:\ \pi(f) = 0, \pi(f^2) = 1\big\}.$$
	The relaxation time $t_{rel}$ is then the reciprocal of $\lambda_2$, that is,
	$$t_{rel} = t_{rel}(Q,\pi) := \dfrac{1}{\lambda_2}.$$
	Note that in the finite state space setting, $\lambda_2$ is the second smallest eigenvalue of $-\bar{Q}$.
	
	\item(Asymptotic variance) For a mean zero function $f$, i.e., $\pi(f)=0$, the central limit theorem for Markov processes \cite[Theorem $2.7$]{KLO12} gives $t^{-1/2} \int_0^t f(X_s)ds$ converges in probability to a mean zero Gaussian distribution with variance
	$$\sigma^2(f,Q,\pi) := -2 \langle f,g \rangle_{\pi},$$
	where $g$ solves the Poisson equation $Qg = f$.

    \item(Large deviation) Let the occupation measure of the Markov chain $X$ be
    $$
    L_t=\frac{1}{t}\int_{0}^{t}\delta_{X_s} ds,
    $$
    and the rate function be
    $$
    I(Q,\nu)=\sup_{u>0}\Big(-\sum_{x\in \mathcal{X}}\nu(x)\frac{Qu(x)}{u(x)}\Big),\quad \nu\in \mathcal{P}(\mathcal{X}),
    $$
    where $\mathcal{P}(\mathcal{X})$ is the set of probability distributions on $\mathcal{X}$. It follows from the large deviation principle that for large $t$ and $A\in \mathcal{P}(\mathcal{X})$,
    $$
    \mathbb{P}(L_t\in A)\approx \text{exp}\left(-t\inf_{\nu\in A}I(Q,\nu)\right).
    $$
	We refer readers to \cite{Ho00} for further references on the subject of large deviations of Markov chains.
	
	\item(Capacity) For any disjoint subset $A,B$ of $\mathcal{X}$, we define the capacity between $A$ and $B$ to be
	$$\mathrm{cap}(A,B) = \mathrm{cap}(A,B,Q,\pi) := \sum_{x \in A} \pi(x) \mathbb{P}_x(\tau_A^+ > \tau_B^+).$$
	If $Q$ is reversible with respect to $\pi$, the classical Dirichlet principle for capacity gives
	$$\mathrm{cap}(A,B) = \inf\{\langle -Qf,f \rangle_{\pi}:\ f|_{A} = 1, f|_{B} = 0\}.$$
	Note that in \cite{Do94} and \cite{GL14}, the authors derived the Dirichlet principle for non-reversible Markov chains.

\end{itemize}

With the above parameters in mind, we are now ready to state our first comparison result between $M_1$ and $M_2$:
\begin{theorem}[Comparison theorem between $M_1(Q,\mu)$ and $M_2(Q,\mu)$]\label{thm:compareM1QmuM2Qmu}
	Given a target distribution $\mu$ on finite state space $\mathcal{X}$ and proposal irreducible chain with generator $Q$, we have the following comparison results between $M_1 = M_1(Q,\mu)$ and $M_2 = M_2(Q,\mu)$:
	\begin{enumerate}
		\item(Hitting times)\label{it:hit} For $\lambda > 0$ and $A \subseteq \mathcal{X}$, we have
		$$\mathbb{E}_{\mu}(e^{-\lambda \tau_A(M_1)}) \leq \mathbb{E}_{\mu}(e^{-\lambda \tau_A(M_2)}).$$
		In particular, $\mathbb{E}_{\mu}(\tau_A(M_1)) \geq \mathbb{E}_{\mu}(\tau_A(M_2))$. Furthermore, $t_{av}(M_1,\mu) \geq t_{av}(M_2,\mu)$.
		
		\item(Total variation mixing time)\label{it:tmix} There exists a positive constant $C_{\mu}$ that depends on $\mu$ such that
		$$t_{mix}(M_2,\mu,1/4) \leq C_{\mu}t_{mix}(M_1,\mu,1/4).$$
		That is, $t_{mix}(M_2,\mu,1/4) \lesssim_{\mu} t_{mix}(M_1,\mu,1/4).$

      \item(Spectral gap)\label{it:l2} We have $\lambda_2(M_1,\mu)\leq \lambda_2(M_2,\mu)$. That is, the exponential $\ell^2$-convergence rate of chain $M_2$ is faster than that of chain $M_1$, or $t_{rel}(M_1,\mu)  \geq t_{rel}(M_2,\mu).$

      \item(Asymptotic variance)\label{it:asymvar} For $h \in \ell^2_0(\mu) = \{h;~\mu(h) = 0\}$,
		$$\sigma^2(h,M_1,\mu) \geq \sigma^2(h,M_2,\mu).$$

      \item(Large deviations)\label{it:larde} For any $\nu \in \mathcal{P}(\mathcal{X})$, $I(M_1,\nu)\leq I(M_2,\nu)$. That is, the deviations for chain $M_2$ from the invariant distribution are asymptotically less likely than for $M_1$.
		
	  \item(Capacity)\label{it:cap} For any disjoint $A,B \subseteq \mathcal{X}$,
		$$\mathrm{cap}(A,B,M_1,\mu) \leq \mathrm{cap}(A,B,M_2,\mu).$$
      In particular, if we take $A=\{x\}$ and $B=\{y\}$, we have
	  $$
      t_{com}^{x,y}(M_1) \geq t_{com}^{x,y}(M_2).
      $$
		
	\end{enumerate}
\end{theorem}

Theorem \ref{thm:compareM1QmuM2Qmu} shows that $M_2$ has superior mixing properties than $M_1$ in almost all aspects: $M_2$ has smaller mean hitting time, average hitting time, commute time, total variation mixing time, relaxation time and asymptotic variance and larger rate function and capacity between any two disjoint sets. As a result, it seems to suggest that for Markov chain Monte Carlo purpose one should use $M_2$ whenever possible since it is faster than its classical Metropolis-Hastings counterpart $M_1$. In Section \ref{sec:MIS}, we offer an explicit spectral analysis for both $M_1$ and $M_2$ in the Metropolised independent sampling setting.

In the next result, we take $\mu = \pi$, the stationary distribution of the proposal chain $Q$ and offer comparison results between $M_1, M_2$ and $\bar{Q}$. Recall that all these generators minimize the distance between $Q$ and $\mathcal{R}(\pi)$ in Theorem \ref{thm:geomM1M2}. They however behave differently based on different parameters:

\begin{theorem}[Comparison theorem between $M_1$, $M_2$ and $\bar{Q}$]\label{thm:compareM1M2additive}
	Given a target distribution $\pi$ on finite state space $\mathcal{X}$ and a proposal chain with generator $Q$. If $\pi$ is the stationary distribution of chain $Q$, then we have the following comparison results between $M_1 = M_1(Q,\pi)$, $M_2 = M_2(Q,\pi)$ and $\bar{Q} = (Q+Q^*)/2$:
	\begin{enumerate}
		\item(Hitting times)\label{hit} For $\lambda > 0$ and $A \subseteq \mathcal{X}$, we have
		$$
		\mathbb{E}_{\pi}(e^{-\lambda \tau_A(M_1)}) \leq \mathbb{E}_{\pi}(e^{-\lambda \tau_A(\bar{Q})}) \leq \min\big\{\mathbb{E}_{\pi}(e^{-\lambda \tau_A(Q)}),\ \mathbb{E}_{\pi}(e^{-\lambda \tau_A(M_2)})\big\}.
        $$
		In particular, for any $A\subseteq\mathcal{X}$,
      $$
      \mathbb{E}_{\pi}(\tau_A(M_1)) \geq \mathbb{E}_{\pi}(\tau_A(\bar{Q})) \geq \max\big\{\mathbb{E}_{\pi}(\tau_A(Q)),\ \mathbb{E}_{\pi}(\tau_A(M_2))\big\}.
      $$
       Furthermore,
		$$
		t_{av}(M_1,\pi) \geq t_{av}(\bar{Q},\pi)\geq \max\big\{t_{av}(Q,\pi),\  t_{av}(M_2,\pi)\big\}.
		$$
		
		\item(Total variation mixing time) There exists positive constants $C_{\pi}^{(1)}, C_{\pi}^{(2)}$ that depend on $\pi$ such that
		$$t_{mix}(M_2,\pi,1/4) \leq C_{\pi}^{(1)}t_{mix}(\bar{Q},\pi,1/4) \leq C_{\pi}^{(2)}t_{mix}(M_1,\pi,1/4).$$
		That is, $t_{mix}(M_2,\pi,1/4) \lesssim_{\pi} t_{mix}(\bar{Q},\pi,1/4) \lesssim_{\pi} t_{mix}(M_1,\pi,1/4).$

      \item(Spectral gap) We have
      $$
      \lambda_2(M_1,\pi)\leq \lambda_2(\bar{Q},\pi)\leq \lambda_2(M_2,\pi).
      $$
      That is, $t_{rel}(M_1,\pi)\geq t_{rel}(\bar{Q},\pi)\geq t_{rel}(M_2,\pi).$

      \item(Asymptotic variance)\label{asva} For $h \in \ell^2_0(\pi) = \{h;~\pi(h)=0\}$,
		$$
		\sigma^2(h,M_1,\pi) \geq \sigma^2(h,\bar{Q},\pi) \geq \max\big\{\sigma^2(h,Q,\pi),\ \sigma^2(h,M_2,\pi)\big\}.
		$$

      \item(Large deviations)\label{lade} For any $\nu\in\mathcal{P}(\mathcal{X})$,
          $$
          I(M_1,\nu)\leq I(\bar{Q},\nu)\leq \min\big\{I(Q,\mu),\ I(M_2,\nu)\big\}.
          $$
		
		\item(Capacity)\label{capa} For any disjoint $A,B \subseteq \mathcal{X}$,
		$$
        \mathrm{cap}(A,B,M_1,\pi) \leq \mathrm{cap}(A,B,\bar{Q},\pi) \leq \min\big\{\mathrm{cap}(A,B,Q,\pi),\ \mathrm{cap}(A,B,M_2,\pi)\big\}.
        $$
      In particular, if we take $A=\{x\}$ and $B=\{y\}$, we have
		$$
        t_{com}^{x,y}(M_1) \geq t_{com}^{x,y}(\bar{Q}) \geq \max\big\{t_{com}^{x,y}(Q),\  t_{com}^{x,y}(M_2)\big\}.
        $$

	\end{enumerate}
\end{theorem}

\begin{rk}
	In Theorem \ref{thm:compareM1M2additive}, we provide a comparison theorem between between $M_1$, $M_2$ and $\bar{Q}$ based on different parameters. We believe that similar results should hold between $Q$ and $M_2$, and conjecture that $M_2$ should mix faster than $Q$ simply because we are taking the maximum between $Q$ and $Q^*$ for off-diagonal entries. We are not able to prove it however due to the non-reversibility of $Q$.
\end{rk}

The rest of this section is devoted to the proof of Theorem \ref{thm:compareM1QmuM2Qmu} and Theorem \ref{thm:compareM1M2additive}.

\subsection{Proof of Theorem \ref{thm:compareM1QmuM2Qmu}}

Many of our results follow from the Peskun ordering between $M_1$ and $M_2$ as well as Lemma \ref{lem:M1M2}. We first prove item \eqref{it:hit}. As $M_2 \succeq M_1$ and both are $\mu$-reversible, \cite[Theorem $3.1$]{HM18} gives the desired result on the Laplace transform order of hitting time.

%Next, we prove item \eqref{it:tav}. Under the additional assumption that $\mathcal{X}$ is finite, we take $A = \{x\}$ for $x \in \mathcal{X}$ in item \eqref{it:hit} that leads to
%$$\mathbb{E}_{\mu}(\tau_x(M_1)) \geq \mathbb{E}_{\mu}(\tau_x(M_2)).$$
%Let $D(M_1,\mu)$ (resp.~$D(M_2,\mu)$) be the deviation matrix of $M_1$ (resp.~$M_2$) with entries given by, for $x,y \in \mathcal{X}$,
%$$D(M_1,\mu)(x,y) := \int_0^{\infty} e^{M_1t}(x,y) - \mu(y) dt, \quad D(M_2,\mu)(x,y) := \int_0^{\infty} e^{M_2t}(x,y) - \mu(y) dt.$$
%Note that according to \cite[equation $(5.7)$]{CvD02}, we then have
%$$D(M_1,\mu)(x,x) = \mu(x)\mathbb{E}_{\mu}(\tau_x(M_1)) \geq \mu(x) \mathbb{E}_{\mu}(\tau_x(M_2)) = D(M_2,\mu)(x,x),$$
%and so
%$$t_{av}(M_2,\mu) = \sum_x D(M_2,\mu)(x,x) \leq \sum_x D(M_1,\mu)(x,x) = t_{av}(M_1,\mu),$$
%where the equality follows from the random target lemma, see e.g. Lemma $10.1$ in \cite{LPW09}.

Next, we prove item \eqref{it:tmix}. First, we fix $A \subseteq \mathcal{X}$ such that $\mu(A) \geq 1/4$, and by item \eqref{it:hit}, we have
$$\mathbb{E}_{\mu}(\tau_A(M_2)) \leq \mathbb{E}_{\mu}(\tau_A(M_1)) \leq \sup_x  \mathbb{E}_{x}(\tau_A(M_1)) \leq  \sup_{x,A: \mu(A) \geq 1/4}  \mathbb{E}_{x}(\tau_A(M_1)).$$
By \cite[Theorem $1.3$]{Oliveria12}, there exists universal constant $C^{(1)}$ such that $\sup_{x,A: \mu(A) \geq 1/4}  \mathbb{E}_{x}(\tau_A(M_1)) \leq C^{(1)} t_{mix}(M_1,\mu,1/4).$ On the other hand, let $x^* := \arg \max \mathbb{E}_{x}(\tau_A(M_2))$ and $\mu_{min} := \min_x \mu(x)$, we then have
$$\mu(x^*) \mathbb{E}_{x^*}(\tau_A(M_2)) \leq \mathbb{E}_{\mu}(\tau_A(M_2)) \leq C^{(1)} t_{mix}(M_1,\mu,1/4),$$
which becomes
$$\sup_{x,A: \mu(A) \geq 1/4}  \mathbb{E}_{x}(\tau_A(M_2)) \leq \dfrac{C^{(1)}}{\mu_{min}} t_{mix}(M_1,\mu,1/4).$$
Using again \cite[Theorem $1.3$]{Oliveria12}, there exists universal constant $C^{(2)}$ such that $$\sup_{x,A: \mu(A) \geq 1/4}  \mathbb{E}_{x}(\tau_A(M_2)) \geq C^{(2)} t_{mix}(M_2,\mu,1/4).$$
Desired result follows from taking $C_{\mu} = \frac{C^{(1)}}{C^{(2)}\mu_{min}}$.

Now, we prove item \eqref{it:l2}. Using the definition of spectral gap and Lemma \ref{lem:M1M2}, we see that
\begin{align*}
\lambda_2(M_2,\mu) &= \inf\{\langle -M_2f,f \rangle_{\mu}:\ \mu(f) = 0, \mu(f^2) = 1\} \\
&\geq \inf\{\langle -M_1f,f \rangle_{\mu}:\ \mu(f) = 0, \mu(f^2) = 1\} \\
&= \lambda_2(M_1,\mu).
\end{align*}
From \cite[Chapter 9]{Chenb92}, this leads to
\begin{align*}
\sup_{||f||_{\ell^2(\mu)} \leq 1} ||e^{M_2t}f - \mu(f)||_{\ell^2(\mu)} = e^{-\lambda_2(M_2,\mu)t} &\leq e^{-\lambda_2(M_1,\mu)t} = \sup_{||f||_{\ell^2(\mu)} \leq 1} ||e^{M_1t}f - \mu(f)||_{\ell^2(\mu)}.\\
\end{align*}

For item \eqref{it:asymvar}, it readily follows from \cite[Theorem $6$]{LM08}. We proceed to prove item \eqref{it:larde}. Denote $R=M_2-M_1$. It is easy to see that $R$ is also a $\mu$-reversible generator. Since $M_1$, $M_2$ and $R$ are $\mu$-reversible generators, from \cite[Theorem IV.14]{Ho00},
\begin{align*}
I(M_2,\nu)&=-\sum_{x,y}\sqrt{\frac{\nu(x)}{\mu(x)}}\mu(x)M_2(x,y)\sqrt{\frac{\nu(y)}{\mu(y)}}\\
&=-\sum_{x,y}\sqrt{\frac{\nu(x)}{\mu(x)}}\mu(x)\Big(M_1(x,y)+R(x,y)\Big)\sqrt{\frac{\nu(y)}{\mu(y)}}\\
&\geq I(M_1,\nu), \quad \text{for }\nu\in \mathcal{P}(\mathcal{X}).
\end{align*}

Finally, we prove item \eqref{it:cap}. We use again Lemma \ref{lem:M1M2} and the Dirichlet principle for capacity of reversible Markov chains in \cite[Chapter 2, Theorem 6.1]{Liggb85} to give
\begin{align*}
\mathrm{cap}(A,B,M_1,\mu) &= \inf\big\{\langle -M_1f,f \rangle_{\mu}:\  f|_{A} = 1, f|_{B} = 0\big\} \\
&\leq \inf\big\{\langle -M_2f,f \rangle_{\mu}:\ f|_{A} = 1, f|_{B} = 0\big\} \\
&= \mathrm{cap}(A,B,M_2,\mu).
\end{align*}
In particular, taking $A = \{x\}$ and $B = \{y\}$, we have
$$\mu(x)\mathbb{P}_x(\tau_y(M_1) < \tau_x^+(M_1)) \leq \mu(x)\mathbb{P}_x(\tau_y(M_2) < \tau_x^+(M_2)).$$
Together with $|M_1(x,x)| \leq |M_2(x,x)|$ and \cite[Chapter $2$ Corollary $8$]{AF14}, desired result follows since
\begin{align*}
t_{com}^{x,y}(M_2) &= \dfrac{1}{|M_2(x,x)|\mu(x)\mathbb{P}_x(\tau_y(M_2) < \tau_x^+(M_2))}\\
&\leq \dfrac{1}{|M_1(x,x)|\mu(x)\mathbb{P}_x(\tau_y(M_1) < \tau_x^+(M_1)) }\\
&= t_{com}^{x,y}(M_1).
\end{align*}
\quad $\square$

\subsection{Proof of Theorem \ref{thm:compareM1M2additive}}
Since $M_1,\ \bar{Q},\ M_2$ are $\pi$-reversible and $M_2\succeq \bar{Q}\succeq M_1$, the proof of the results for them is omitted as it is essentially the same as the proof of Theorem \ref{thm:compareM1QmuM2Qmu} with $\mu$ substituted by $\pi$. It remains to prove the results for chain $Q$.

For hitting times, from \cite[Theorem 3.3]{HM18} it follows that
$$
\mathbb{E}_{\pi}(e^{-\lambda \tau_A(\bar{Q})}) \leq \mathbb{E}_{\pi}(e^{-\lambda \tau_A(Q)})\quad \text{and}\quad \mathbb{E}_\pi(\tau_A(\bar{Q}))\geq \mathbb{E}_\pi(\tau_A(Q)),
$$
for any $A\subseteq\mathcal{X}$. Hence, item \eqref{hit} holds. Next, we prove item \eqref{asva}. In fact, applying \cite[Theorem 4.3]{HM18.2} to the continuous-time case, i.e., replacing $I-P$ by $Q$ in its proof, gives that
$$
\sigma^2(h,\bar{Q},\pi) \geq \sigma^2(h,Q,\pi),\quad \text{for}\ h\in \ell^2_0(\pi).
$$

For large deviations, \cite[Proposition 3.2]{Bie16} gives the desired result. Finally, we use the Dirichlet principle of capacity in \cite{GL14} to prove item \eqref{capa}. More specifically, from \cite[Lemma 3.2]{GL14}, we can see that

\begin{align*}
\text{cap}(A,B,Q,\pi)&=\inf_{f|_A=1,f|_B=0}\sup_{g|_A,g|_B \text{constants}} \big\{2\langle Q^*f,g \rangle_\pi-\langle g,(-\bar{Q})g\rangle_\pi\big\} \\
&\geq \inf_{f|_A=1,f|_B=0}\langle f, (-\bar{Q})f\rangle_\pi\\
&=\text{cap}(A,B,\bar{Q},\pi),
\end{align*}
where we take $g=-f$ in the inequality.
\quad $\square$

\section{Examples}\label{sec:ex}

In this section, two examples are provided to illustrate our main results. In the first example in Section \ref{sec:MIS}, we give an eigenanalysis for the case of Metropolised independing sampling, while in the second example in Section \ref{sec:bd}, we consider the case when $Q$ is a birth-death chain and compare the fastest strong stationary time of $M_1$ and $M_2$.

\subsection{Metropolised independent sampling and spectral analysis of $M_2$}\label{sec:MIS}
In this section, we offer an explicit spectral analysis for both $M_1$ and $M_2$ of Metroplised independent sampling on a finite state space $\mathcal{X} = \llbracket 1,m \rrbracket = \{1,2,\ldots,m\}$ with $m \in \mathbb{N}$. This section is inspired by the work of \cite{Liu96} who offered the first explicit eigenanalysis of $M_1$ for Metropolised independent sampling. We will show that similar results can be obtained for $M_2$ using the techniques therein. Suppose that $\mathbf{p} = (p_y)_{y \in \mathcal{X}}$ is a probability distribution on $\mathcal{X}$, and denote by $P = (P(x,y))_{x,y \in \mathcal{X}}$ with $P(x,y) = p_y$ to be a transition matrix of the form
$$P = \begin{bmatrix}
p_1 & p_2 & \ldots & p_m  \\
\vdots & \vdots & \vdots & \vdots \\
p_1 & p_2 & \ldots & p_m
\end{bmatrix}.$$
In addition, we take the proposal chain to be the continuized chain of $P$ with generator $Q := P - I$, where $I$ is the identity matrix of size $m \times m$. For a given target distribution $\mu = (\mu(x))_{x \in \mathcal{X}}$, we define
$$w_x := \dfrac{\mu(x)}{p_x}$$
and assume without loss of generality (by relabelling the state space) that $w_1 \geq w_2 \geq \ldots \geq w_m$. As a result, both $M_1(Q,\mu)$ and $M_2(Q,\mu)$ take the form
\begin{align*}
M_1(Q,\mu) &= \begin{bmatrix}
p_1 + \gamma_1-1 & \frac{\mu(2)}{w_1} & \frac{\mu(3)}{w_1} & \ldots & \frac{\mu(m)}{w_1}  \\
p_1 & p_2 + \gamma_2 - 1 & \frac{\mu(3)}{w_2} & \ldots & \frac{\mu(m)}{w_2} \\
p_1 & p_2 & p_3 + \gamma_3 - 1 & \ldots & \frac{\mu(m)}{w_3} \\
\vdots & \vdots & \vdots & \vdots & \vdots \\
p_1 & p_2 & p_3 & \ldots & p_m - 1
\end{bmatrix}, \\
M_2(Q,\mu) &= \begin{bmatrix}
p_1-1 & p_2 & p_3 & \ldots & p_m  \\
\frac{\mu(1)}{w_2} & p_2 + \beta_2 - 1 & p_3 & \ldots & p_m \\
\frac{\mu(1)}{w_3} & \frac{\mu(2)}{w_3} & p_3 + \beta_3 - 1 & \ldots & p_m \\
\vdots & \vdots & \vdots & \vdots & \vdots \\
\frac{\mu(1)}{w_m} & \frac{\mu(2)}{w_m} & \frac{\mu(3)}{w_m} & \ldots & p_m + \beta_m - 1
\end{bmatrix},
\end{align*}
where for $x \in \llbracket 1,m-1 \rrbracket$ and $i \in \llbracket 2,m \rrbracket$,
$$\gamma_x := \sum_{j = x}^{m} \dfrac{\mu(j)}{w_j} - \dfrac{\mu(j)}{w_x} \geq 0, \quad \beta_i := \sum_{j=1}^i \dfrac{\mu(j)}{w_j} - \dfrac{\mu(j)}{w_i} \leq 0.$$

In our result below, we show that $(\beta_i-1)_{i \in \llbracket 2,m \rrbracket}$ (resp.~$(\gamma_x-1)_{x \in \llbracket 1,m-1 \rrbracket}$) are the eigenvalues of $M_2(Q,\mu)$ (resp.~$M_1(Q,\mu)$).

\begin{proposition}[Eigenanalysis of $M_1$ and $M_2$ for Metropolised independent sampling]\label{prop:eigenM1M2}
	Given a target distribution $\mu$ on $\mathcal{X} = \llbracket 1,m \rrbracket$ and proposal chain with generator $Q = P - I$, the non-zero eigenvalues-eigenvectors of $M_2(Q,\mu)$ are $(\beta_i-1,\mathbf{v}_i)_{i \in \llbracket 2,m \rrbracket}$, while that of $M_1(Q,\mu)$ are $(\gamma_x-1,\mathbf{w}_x)_{x \in \llbracket 1,m-1 \rrbracket}$,
	where
	\begin{align*}
		\mathbf{v}_i &= \bigg(-\mu(i),-\mu(i),\ldots,-\mu(i),\underbrace{\sum_{j \leq i-1} \mu(j)}_{\text{i}^{th} position},0,\ldots,0 \bigg)^T, \\
		\mathbf{w}_x &= \bigg(0,0,\ldots,0,\underbrace{\sum_{j \geq x+1} \mu(j)}_{\text{x}^{th} position},-\mu(x),\ldots,-\mu(x) \bigg)^T.
	\end{align*}
\end{proposition}

\begin{rk}
	In fact, the spectral information of $M_1(Q,\mu)$ can be obtained from $M_2(Q,\mu)$ by reordering the index in the way of changing $i$ to $m-i+1$ and replacing $\beta$ by $\gamma$. 
\end{rk}

\begin{rk}
	As explicit eigenvalues information are available for $M_2(Q,\mu)$, by means of Theorem \ref{thm:compareM1QmuM2Qmu} item \eqref{it:l2} we have
	$$e^{(\max_{i \in \llbracket 2,m \rrbracket} \beta_i-1)t} = \sup_{||f||_{\ell^2(\mu)} \leq 1} ||e^{M_2t}f - \mu(f)||_{\ell^2(\mu)} \leq \sup_{||f||_{\ell^2(\mu)} \leq 1} ||e^{M_1t}f - \mu(f)||_{\ell^2(\mu)} = e^{(\gamma_1-1)t}.$$
\end{rk}

\subsection{Proof of Proposition \ref{prop:eigenM1M2}}

In this section, we prove Proposition \ref{prop:eigenM1M2} for $M_2$ by adapting similar techniques as in \cite{Liu96}. The case for $M_1$ has already been done in \cite[Theorem $2.1$]{Liu96}. First, we denote $\mathbf{1} := (1,1,\ldots,1)^T$ to be the column vector of all ones and $G$ by
$$G := M_2(Q,\mu) -Q = \begin{bmatrix}
0 & 0 & 0 & \ldots & 0  \\
\frac{\mu(1)}{w_2}-p_1 & \beta_2 & 0 & \ldots & 0 \\
\frac{\mu(1)}{w_3}-p_1 & \frac{\mu(2)}{w_3}-p_2 & \beta_3 & \ldots & 0 \\
\vdots & \vdots & \vdots & \vdots & \vdots \\
\frac{\mu(1)}{w_m}-p_1 & \frac{\mu(2)}{w_m}-p_2 & \frac{\mu(3)}{w_m}-p_3 & \ldots & \beta_m
\end{bmatrix}.$$
Note that $\mathbf{1}$ is a common right eigenvector of both $M_2(Q,\mu) + I$ and $M_2(Q,\mu) + I - G = P$ with eigenvalue $1$. Since $M_2(Q,\mu) + I - G = P$ is of rank one, the rest of the eigenvalues of $M_2(Q,\mu) + I$ and $G$ have to be the same, and hence the non-zero eigenvalues for $M_2(Q,\mu)$ are $(\beta_i-1)_{i \in \llbracket 2,m \rrbracket}$. To determine the eigenvectors, we begin with the eigenvectors for $G$.

\begin{lemma}
	For $i \in \llbracket 2,m \rrbracket$,
	$$\mathbf{u}_i = (\mathbf{u}_i(x))_{x \in \llbracket 1,m \rrbracket} =  \bigg(0,0,\ldots,0,\underbrace{\sum_{j \leq i} \mu(j)}_{\text{i}^{th} position},\mu(i),\ldots,\mu(i) \bigg)^T$$
	is an eigenvector associated with eigenvalue $\beta_i$ of $G$.
\end{lemma}

\begin{proof}
	We consider the $k$-th entry of the vector $G\mathbf{u}_i$, denoted by $G\mathbf{u}_i(k)$, for the following three cases $k < i$, $k = i$ and $k > i$. In the first case when $k < i$,
	$$G\mathbf{u}_i(k) = 0 = \beta_i \mathbf{u}_i(k).$$	
	In the second case when $k = i$, we have
	$$G\mathbf{u}_i(k) = \beta_i \sum_{j \leq i} \mu(j) = \beta_i \mathbf{u}_i(k).$$	
	Finally, in the last case when $k > i$, we check that
	\begin{align*}
	G\mathbf{u}_i(k) &= G(k,i) \sum_{j \leq i} \mu(j) + \mu(i) \sum_{j = i+1}^k G(k,j)\\
	&= \left(\dfrac{\mu(i)}{w_k} - p_i\right) \sum_{j \leq i} \mu(j) + \mu(i) \sum_{j = i+1}^{k-1} \left( \dfrac{\mu(j)}{w_k} - \dfrac{\mu(j)}{w_j} \right) + \mu(i) \beta_k \\
	&= \mu(i)\left(\sum_{j=1}^i \dfrac{\mu(j)}{w_j} - \dfrac{\mu(j)}{w_i}\right) = \beta_i \mu(i) = \beta_i \mathbf{u}_i(k).
	\end{align*}
\end{proof}

We now proceed to show that $\mathbf{v}_i = \mathbf{u}_i - \mu(i) \mathbf{1}$ is an eigenvector. First, we note that
$$\mathbf{p} \mathbf{u}_i = p_i \sum_{j \leq i} \mu(j) + \mu(i) \sum_{j=k+1}^m p_j = \mu(i) (1-\beta_i),$$
and so
$$\left(M_2(Q,\mu) + I \right)\mathbf{u}_i = G \mathbf{u}_i + \mathbf{1} \mathbf{p} \mathbf{u}_i = \beta_i \mathbf{u}_i + \mu(i) (1-\beta_i) \mathbf{1}.$$
Since $\mathbf{1}$ is a right eigenvector of $M_2(Q,\mu) + I$ with eigenvalue $1$, for any $t$ we have
$$\left(M_2(Q,\mu) + I\right)(\mathbf{u}_i - t \mathbf{1}) = \beta_i \left(\mathbf{u}_i - \dfrac{t - \mu(i)(1-\beta_i)}{\beta_i}\mathbf{1}\right).$$
Solving for $t = \dfrac{t - \mu(i)(1-\beta_i)}{\beta_i}$ gives $\mathbf{v}_i = \mathbf{u}_i - \mu(i) \mathbf{1}$ is an eigenvector.

\subsection{Comparing the fastest strong stationary time of birth-death Metropolis chains $M_1$ and $M_2$}\label{sec:bd}

In our second example, we consider the case when the state space is $\mathcal{X} = \llbracket 0,n \rrbracket$ and the proposal chain $Q$ is an ergodic birth-death chain, that is, $Q(x,y) > 0$ if and only if $|y-x| = 1$. In this setting, it is easy to see that both $M_1(Q,\mu)$ and $M_2(Q,\mu)$ are birth-death chains. For this example we are primarily interested in the so-called fastest strong stationary time $T_{sst}(Q)$ starting from $0$, which is a randomized stopping time that satisfies, for $t > 0$,
$$\max_{y \in \mathcal{X}} \bigg\{1- \dfrac{e^{Qt}(0,y)}{\pi(y)}\bigg\} = \mathbb{P}_0(T_{sst}(Q) > t).$$
As $Q$ is an birth-death chain, classical results (see e.g. \cite[Corollary $4.2$]{DSC06} or \cite[Theorem $1.4$]{Fill}) tell us that $T_{sst}(Q)$ under $\mathbb{P}_0$ has the law of convolution of exponential distributions with parameters being the non-zero eigenvalues of $-Q$. This fact, together with the Courant-Fischer min-max theorem of eigenvalues, give rise to the following comparison results:

\begin{proposition}[Fastest strong stationary time of birth-death Metropolis chains $M_1$ and $M_2$]
	Given a target distribution $\mu$ on finite state space $\mathcal{X}$ and birth-death proposal chain with generator $Q$, by writing the eigenvalues of $-M_1 = -M_1(Q,\mu)$ and $-M_2 = -M_2(Q,\mu)$ in ascending order for $j = 1,2$
	$0 = \lambda_1(M_j) \leq \lambda_2(M_j) \ldots \leq \lambda_{|\mathcal{X}|}(M_j),$
	we have, for $i \in \llbracket 1,|\mathcal{X}| \rrbracket$,
	$$\lambda_i(M_1) \leq \lambda_i(M_2).$$
	Consequently, there is a Laplace transform order of the fastest strong stationary time starting at $0$ between $M_1$ and $M_2$, that is, for $\alpha > 0$,
	$$\mathbb{E}_0(e^{-\alpha T_{sst}(M_2)}) \geq \mathbb{E}_0(e^{-\alpha T_{sst}(M_1)}).$$
	In particular, the mean and variance of the fastest strong stationary time are ordered by
	\begin{align*}
		\mathbb{E}_0(T_{sst}(M_2)) &\leq \mathbb{E}_0(T_{sst}(M_1)), \\
		\mathrm{Var}_0(T_{sst}(M_2)) &\leq \mathrm{Var}_0(T_{sst}(M_1)).
	\end{align*}
\end{proposition}

\begin{proof}
	First, by Lemma \ref{lem:M1M2} we have $\langle -M_2 f,f \rangle_{\mu} \geq \langle -M_1 f,f \rangle_{\mu}$. By the Courant-Fischer min-max theorem of eigenvalues \cite[Theorem $4.2.6$]{HJ13}, this leads to
	$$\lambda_i(M_1) \leq \lambda_i(M_2).$$
	Consequently, for $\alpha > 0$,
	$$\dfrac{\lambda_i(M_1)}{\lambda_i(M_1) + \alpha} \leq \dfrac{\lambda_i(M_2)}{\lambda_i(M_2) + \alpha}.$$
	Taking the product yields
	$$\mathbb{E}_0(e^{-\alpha T_{sst}(M_2)}) = \prod_{i=2}^{|\mathcal{X}|} \dfrac{\lambda_i(M_2)}{\lambda_i(M_2) + \alpha} \geq \prod_{i=2}^{|\mathcal{X}|} \dfrac{\lambda_i(M_1)}{\lambda_i(M_1) + \alpha} = \mathbb{E}_0(e^{-\alpha T_{sst}(M_1)}),$$
	where the equalities follow from the fact that both $M_1$ and $M_2$ are birth-death chains and so their fastest strong stationary times are distributed as convolution of exponential distributions with parameters being the non-zero eigenvalues. In particular, we have
	\begin{align*}
	\mathbb{E}_0(T_{sst}(M_2)) = \sum_{i=2}^{|\mathcal{X}|} \dfrac{1}{\lambda_i(M_2)} &\leq \sum_{i=2}^{|\mathcal{X}|} \dfrac{1}{\lambda_i(M_1)} = \mathbb{E}_0(T_{sst}(M_1)), \\
	\mathrm{Var}_0(T_{sst}(M_2)) = \sum_{i=2}^{|\mathcal{X}|} \dfrac{1}{\lambda_i(M_2)^2} &\leq \sum_{i=2}^{|\mathcal{X}|} \dfrac{1}{\lambda_i(M_1)^2} = \mathrm{Var}_0(T_{sst}(M_1)).
	\end{align*}
\end{proof}

\noindent \textbf{Acknowledgements}.
The authors would like to thank the anonymous referee for constructive comments that improve the presentation of the manuscript. Michael Choi acknowledges the support from the Chinese University of Hong Kong, Shenzhen grant PF01001143. Lu-Jing Huang acknowledges the support from NSFC No.11771047 and Probability and Statistics: Theory and Application(IRTL1704). The authors would also like to thank Professor Yong-Hua Mao for his hospitality during their visit to Beijing Normal University, where this work was initiated.

\bibliographystyle{abbrvnat}
\bibliography{thesis}

\end{document}